\documentclass[11pt,a4paper,oneside,reqno]{amsart}
\usepackage[T1]{fontenc}
\usepackage[utf8]{inputenc}

\usepackage{lmodern}
\usepackage{microtype}

\usepackage[english]{babel}

\usepackage{amsthm}
\usepackage{amsmath}
\usepackage{amsfonts} 
\usepackage{amssymb}
\usepackage{mathrsfs}
\usepackage[inline]{enumitem}

\usepackage{esint}

\usepackage[hcentering]{geometry}

\usepackage{ifdraft}

\ifdraft{\usepackage[default,angular,scale=2, scaled]{comicneue}\usepackage{sfmath}
\linespread{1.5}\newgeometry{margin=1cm}\usepackage{relsize}\AtBeginDocument{\relscale{1.3}}}{}

\usepackage[unicode=true, pdflang={en-GB}, pdfusetitle]{hyperref}

\usepackage[abbrev,backrefs]{amsrefs}

\usepackage[capitalize]{cleveref}

\numberwithin{equation}{section}

\newtheorem{proposition}{Proposition}[section]

\newtheorem{theorem}[proposition]{Theorem}
\newtheorem{lemma}[proposition]{Lemma}

\theoremstyle{definition}

\newcommand{\defeq}{\coloneqq}

\newcommand{\Nset}{\mathbb{N}}
\newcommand{\Zset}{\mathbb{Z}}
\newcommand{\Rset}{\mathbb{R}}

\newcommand{\Sset}{\mathbb{S}}

\newcommand{\Bset}{\mathbb{B}}

\newcommand{\dif}{\,\mathrm{d}}
\usepackage{mathtools}

\newcommand{\compose}{\,\circ\,}
\newcommand{\manifold}[1]{\mathcal{#1}}
\newcommand{\lifting}[1]{\smash{\widetilde{#1}}}

\DeclarePairedDelimiter{\brk}{(}{)}
\DeclarePairedDelimiter{\abs}{\lvert}{\rvert}

\DeclarePairedDelimiter{\seminorm}{\lvert}{\rvert}
\DeclarePairedDelimiter{\floor}{\lfloor}{\rfloor}

\DeclarePairedDelimiterX{\intvc}[2]{[}{]}{#1,#2}
\DeclarePairedDelimiterX{\intvl}[2]{(}{]}{#1,#2}
\DeclarePairedDelimiterX{\intvr}[2]{[}{)}{#1,#2}
\DeclarePairedDelimiterX{\intvo}[2]{(}{)}{#1,#2}
\DeclarePairedDelimiterX{\setcond}[2]{\{}{\}}{#1 \,\delimsize\vert\, #2}

\newcommand{\VMO}{\mathrm{VMO}}

\newcommand{\restr}[1]{\vert_{#1}}
\newcommand{\Deriv}{\mathrm{D}}

\DeclareMathOperator{\id}{id}

\newcommand\stSymbol[1][]{%
\nonscript\;#1\vert
\allowbreak
\nonscript\;
\mathopen{}}
\DeclarePairedDelimiterX\set[1]\{\}{%
\renewcommand\st{\stSymbol[\delimsize]}
#1
}
\providecommand{\st}{\stSymbol}

\DeclareMathOperator{\dist}{dist}

\DeclareMathOperator{\tr}{tr}

\newcommand{\sobolev}{\smash{\dot{W}}\vphantom{W}}

\newcommand{\energy}{\mathfrak{E}}

\usepackage{constants}

\renewcommand{\PrintDOI}[1]{%
  \href{http://dx.doi.org/#1}{doi:#1}%
}
\BibSpec{book}{%
    +{}  {\PrintPrimary}                {transition}
    +{,} { \textit}                     {title}
    +{.} { }                            {part}
    +{:} { \textit}                     {subtitle}
    +{,} { \PrintEdition}               {edition}
    +{}  { \PrintEditorsB}              {editor}
    +{,} { \PrintTranslatorsC}          {translator}
    +{,} { \PrintContributions}         {contribution}
    +{,} { }                            {series}
    +{,} { \voltext}                    {volume}
    +{,} { }                            {publisher}
    +{,} { }                            {organization}
    +{,} { }                            {address} 
    +{,} { \PrintDateB}                 {date}
    +{,} { }                            {status}
    +{,} { \PrintDOI}                   {doi}
    +{}  { \parenthesize}               {language}
    +{}  { \PrintTranslation}           {translation}
    +{;} { \PrintReprint}               {reprint}
    +{.} { }                            {note}
    +{.} {}                             {transition}
    +{}  {\SentenceSpace \PrintReviews} {review}
}

\title
{%
On the local character of the extension of traces for Sobolev mappings
}

\author{Jean Van Schaftingen}


\newcommand{\sectref}[1]{\hyperref[#1]{\S \ref*{#1}}}

\begin{document}

\address{
Universit\'e catholique de Louvain, Institut de Recherche en Math\'ematique et Physique, Chemin du Cyclotron 2 bte L7.01.01, 1348 Louvain-la-Neuve, Belgium}

\email{Jean.VanSchaftingen@UCLouvain.be}

\begin{abstract}
We prove that a mapping $u \colon \mathcal{M}'\to \mathcal{N}$, where $\mathcal{M}'$ and $ \mathcal{N}$ are compact Riemannian manifolds,
is the trace of a Sobolev mapping $U \colon  \mathcal{M}' \times [0, 1) \to \mathcal{N}$
if and only if it is on some open covering of $\smash{\manifold{M}'}$.
In the global case where $\mathcal{M}$ is a compact Riemannian manifold with boundary, this implies that the analytical obstructions to the extension of a mapping $u \colon \partial \mathcal{M}\to \mathcal{N}$ to some Sobolev mapping $U \colon  \mathcal{M} \to \mathcal{N}$ are purely local.
\end{abstract}

\thanks{The author was supported by the Projet de Recherche T.0229.21 ``Singular Harmonic Maps and Asymptotics of Ginzburg--Landau Relaxations'' of the Fonds de la Recherche Scientifique--FNRS.}

\maketitle

\setcounter{tocdepth}{1}
\setcounter{tocdepth}{5}

\section{Introduction}

\subsection{Sobolev mappings, their traces and the extensions of the latter}

Given \(1 \le p < +\infty\), the \emph{first-order homogeneous Sobolev space of mappings} is defined as
\begin{equation*}
 \sobolev^{1, p} \brk{\manifold{M}, \mathcal{N}}
 \defeq
 \set[\Big]{U\colon \manifold{M} \to \mathcal{N} \st U \text{ is weakly differentiable and } \int_{\manifold{M}} \abs{\Deriv U}^p < +\infty },
\end{equation*}
where \(\manifold{M}\) is a Riemannian manifold with  boundary and \(\manifold{N}\) is a Riemnnian manifold.
In view of the Nash isometric embedding theorem \cite{Nash_1956}, we can assume without loss of generality that \(\manifold{N}\) is isometrically embedded into \(\Rset^\nu\); the weak differentiability of the mapping \(U\) is understood as a function from \(\manifold{M}\) to \(\Rset^\nu\), through local charts on \(\manifold{M}\);
the measure and the norm are induced by their respective Riemannian metrics.
Sobolev spaces of mappings arise naturally in calculus of variations and partial differential equations in \emph{physical models with nonlinear order parameters} \citelist{\cite{Mermin1979}\cite{Ball_Zarnescu_2011}}, in \emph{Cosserat models in elasticity} \citelist{\cite{Ericksen_Truesdell_1958}}, in the study of \emph{harmonic maps} in geometry \citelist{\cite{Helein_Wood_2008}\cite{Eells_Lemaire_1978}} and in the description of attitudes of geometrical elements
in \emph{computer graphics and meshing} \cite{Huang_Tong_Wei_Bao_2011}.

The Sobolev space \(\sobolev^{1, p} \brk{\manifold{M}, \mathcal{N}}\) is naturally endowed with the metric space structure of its linear counterpart \(\sobolev^{1, p} \brk{\manifold{M}, \Rset^\nu}\) as a subspace of the latter,
from which they inherit many features such as the completeness, Fubini-like desintegration and reintegration properties, the Sobolev embeddings and the chain rule.

\medbreak

As Sobolev spaces of mappings are not linear spaces, the more linear constructions of the linear theory fail for them. For instance, they are \emph{not linear spaces}, and they can even exhibit a rich \emph{structure of homotopy classes} provided the domain and the target display enough topological complexity \citelist{\cite{Brezis_Li_Mironescu_Nirenberg_1999}\cite{Hang_Lin_2003_II}\cite{Brezis_Li_2001}}.

Another delicate question for Sobolev mappings is their \emph{approximation by smooth maps,} as the classical approach of approximating by averaging or convolution is incompatible with the manifold constraint.
Even though when \(p \ge \dim \manifold{M}\), Sobolev mappings in \(\sobolev^{1, p} (\manifold{M}, \mathcal{N})\) are essentially continuous and thus can be approximated by smooth mappings \citelist{\cite{Schoen_Uhlenbeck_1982}*{\S 3}\cite{Schoen_Uhlenbeck_1983}*{\S 4}}, this is not anymore the case when \(p < m\) and when the homotopy group \(\pi_{\floor{p}} (\mathcal{N})\) is nontrivial, where \(\floor{t} \in \Zset\) is the integer part of the real number \(t \in \Rset\) \citelist{\cite{Bethuel_1991}\cite{Hardt_Lin_1987}}.

\medbreak

Yet another subtle question concerns the \emph{theory of traces,} which is the object of the present work.
The classical trace theory for linear Sobolev spaces states that if \(\manifold{M}\) is a Riemannian manifold with compact boundary \(\partial \manifold{M}\) and if \(p \in \intvo{1}{\infty}\),
 then there is a well-defined continuous and surjective trace operator
 \[
  \tr_{\partial \manifold{M}} \colon 
  \sobolev^{1, p} \brk{\manifold{M}, \Rset^\nu}
  \to  \sobolev^{1-1/p, p} \brk{\partial \manifold{M}, \Rset^\nu}
 \]
 such that \(\tr_{\partial \manifold{M}} U = U \restr{\partial \manifold{M}}\) when \(U \in C \brk{\manifold{M}, \Rset^\nu} \cap \sobolev^{1, p} \brk{\manifold{M}, \Rset^\nu}\) \cite{Gagliardo_1957}.
The range of the trace operator is a \emph{fractional Sobolev-Slobodeckiĭ} space; given \(0 \in \intvo{0}{1}\), \(p \in \intvr{1}{\infty}\) and a \(d\)-dimensional Riemannian manifold \(\manifold{W}\), these spaces are defined as
\begin{equation}
\label{eq_ooJubaebeeh4xae8eoPh0tei}
  \sobolev^{s, p} \brk{\manifold{W}, \Rset^\nu}
  \defeq 
  \set[\Big]{u \colon \manifold{W} \to \Rset^\nu \st \energy^{s, p} \brk{u}  < \infty},
\end{equation}
in terms of the Gagliardo energy
\begin{equation}
\energy^{s, p} \brk{u, \manifold{W}} \defeq \iint\limits_{\manifold{W}\times \manifold{W}} \frac{\abs{u \brk{x} - u \brk{y}}^p}{d \brk{x, y}^{sp + d}} \dif x \dif y.
\end{equation}

A straightforward transfer of the trace theory from linear to nonlinear Sobolev spaces merely shows the well-definiteness of the trace operator, the inclusion
\begin{equation}
\label{eq_iocief3mohfae0aiph4aeSha}
  \tr_{\partial \manifold{M}} \brk{
  \sobolev^{1, p} \brk{\manifold{M}, \manifold{N}}}
  \subseteq \sobolev^{1-1/p, p} \brk{\partial \manifold{M}, \manifold{N}},
 \end{equation}
and the estimate,
for every \(u \in  \sobolev^{1-1/p, p} \brk{\partial \manifold{M}, \manifold{N}}\),
\begin{equation}
\label{eq_cheegai3she5uZahwookieTh}
\energy^{1 - 1/p, p} \brk{u, \partial \manifold{M}}
\le C
  \energy^{1, p}_{\mathrm{ext}} \brk{u, \partial \manifold{M}, \manifold{M}}
\end{equation}
where the \emph{Sobolev extension energy} of a Borel-measurable map \(u \colon \partial \manifold{M} \to \manifold{N}\) is defined as 
\[
  \energy^{1, p}_{\mathrm{ext}}
  \brk{u, \partial \manifold{M}, \manifold{M}}
  \defeq 
  \inf \set[\Big]{\int_{\manifold{M}}
  \abs{\Deriv U}^p
  \st U \in \sobolev^{1, p} \brk{\manifold{M}, \manifold{N}} \text{ and }\tr_{\partial \manifold{M}} U= u}.
\]
 The reverse inclusion to \eqref{eq_iocief3mohfae0aiph4aeSha} and inequality to \eqref{eq_cheegai3she5uZahwookieTh} are much more delicate, as
there is no reason for the extension \(U \in  \sobolev^{1, p} \brk{\manifold{M}, \Rset^\nu}\) constructed by  the classical linear theory from \(u \in \sobolev^{1-1/p, p} \brk{\partial \manifold{M}, \manifold{N}}\) through averaging or convolution to take its value in the target manifold \(\mathcal{N}\) almost everywhere in \(\manifold{M}\).

When \(p \ge \dim \manifold{M}\), Sobolev mappings are continuous or almost -- when \(p = \dim \manifold{M}\) they still have \emph{vanishing mean oscillation} \citelist{\cite{Brezis_Nirenberg_1995}\cite{Brezis_Nirenberg_1996}\cite{Abbondandolo_1996}} -- so that the classical extension stays near the target in a neighourhood of the boundary and the question of extension of Sobolev mappings is reduced to the extension of continuous mappings \cite{Bethuel_Demengel_1995}.

If \(p < \dim \manifold{M}\), the extension of traces encounters various obstructions:
\begin{enumerate}[label=\((\mathrm{Ob}_{\arabic*})\)]
 \item 
 \label{it_hiethie2ohb6ieNgai3dai1u}
 if \(2 \le p < \dim \manifold{M}\) and \(\pi_{\floor{p - 1}} (\mathcal{N}) \not \simeq \set{0}\), then \citelist
{\cite{Hardt_Lin_1987}\cite{Bethuel_Demengel_1995}}
 \[   
 \sobolev^{1 - 1/p, p}\brk{\partial \manifold{M}, \manifold{N}}
   \not \subseteq \overline{\vphantom{\dot{W}}\tr_{\partial \manifold{M}} \brk{\sobolev^{1, p} \brk{\manifold{M}, \manifold{N}}}}.
 \]
 \item 
 \label{it_LoNieshosixiazohxohgieg0}
 if \(2 \le p <\dim \manifold{M}\) and if \(\pi_{\ell} \brk{\mathcal{N}}\) is infinite for some \(\ell \in \set{1, \dotsc, \floor{p - 1}}\), then \cite{Bethuel_2014} (see also \cite{Bethuel_Demengel_1995})
 \[
   \overline{\vphantom{\dot{W}}\tr_{\partial \manifold{M}} \brk{\sobolev^{1, p} \brk{\manifold{M}, \manifold{N}}}}
   \not \subseteq \tr_{\partial \manifold{M}} \brk{\sobolev^{1, p} \brk{\manifold{M}, \manifold{N}}},
 \]
 \item 
 \label{it_coReizacoHoofae7euz5vahk}
 if \(p \in \set{2, \dotsc, \dim \manifold{M}}\) and if  \(\pi_{\ell} \brk{\mathcal{N}} \not \simeq \set{0}\), then \cite{Mironescu_VanSchaftingen_2021_AFST}
 \[
   \overline{\vphantom{\dot{W}}\tr_{\partial \manifold{M}} \brk{\sobolev^{1, p} \brk{\manifold{M}, \manifold{N}}}}
   \not \subseteq \tr_{\partial \manifold{M}} \brk{\sobolev^{1, p} \brk{\manifold{M}, \manifold{N}}}.
 \]
\end{enumerate}

\subsection{Extensions to collar neighourhoods}

When the topology of the domain manifold \(\manifold{M}\) is sufficiently simple and if none of the obstructions \ref{it_hiethie2ohb6ieNgai3dai1u}, \ref{it_LoNieshosixiazohxohgieg0} or \ref{it_coReizacoHoofae7euz5vahk} hold then the trace is surjective:
After several results covering large classes of manifolds \citelist{\cite{Hardt_Lin_1987}\cite{Mironescu_VanSchaftingen_2021_AFST}}, 
it has been proved \cite{VanSchaftingen_ExtensionTraces} that if \(\smash{\manifold{M}'}\) is a compact manifold and \(1 < p < \dim \brk{\smash{\manifold{M}'} \times \intvr{0}{1}}\), then
\begin{equation}
\label{eq_AhH4EiNohbaesh8oocei2wua}
  \tr_{\smash{\manifold{M}'}} \brk{
  \sobolev^{1, p} \brk{\smash{\manifold{M}'}\times \intvr{0}{1}, \manifold{N}}}
  = \sobolev^{1-1/p, p} \brk{\smash{\manifold{M}'}, \manifold{N}}.
\end{equation}
if and only if the first homotopy groups \(\pi_{1} \brk{\manifold{N}}, \dotsc, \smash{\pi_{\floor{p - 2}}} \brk{\manifold{N}}\) are all finite and  the homotopy group \(\smash{\pi_{\floor{p - 1}}} \brk{\manifold{N}}\) is trivial;
moreover, one has then 
\begin{equation}
\label{eq_ohd0eequ9ahleuhouYi3beiG}
\energy^{1, p}_{\mathrm{ext}} \brk{u, \smash{\manifold{M}'}, \smash{\manifold{M}'} \times \intvr{0}{1}}
\le C \energy^{1 - 1/p, p} \brk{u, \smash{\manifold{M}'}}.
\end{equation}

A further question is to \emph{characterise} the range of the trace \(\tr_{\smash{\manifold{M}'}} \brk{
  \sobolev^{1, p} \brk{\smash{\manifold{M}'}\times \intvr{0}{1}, \manifold{N}}}\)
  when \eqref{eq_AhH4EiNohbaesh8oocei2wua} fails.
Isobe has first characterised the traces of Sobolev mappings by the \emph{asymptotics of a Ginzburg-Landau like relaxation} \cite{Isobe_2003}*{lem.\ 2.1 and proof of th.\ 1.2}.
In order to describe it, we  consider the penalised energy associated to  a given Borel-measurable function \(F \colon \Rset^\nu \to \intvc{0}{\infty}\)
and defined for each Borel-measurable function \(u \colon \partial \manifold{M} \to \Rset^\nu\) as
\begin{equation}
\label{eq_re0iP9aeShiejie5eey3ahxe}
  \energy^{1, p}_{F}
  \brk{u, \partial \manifold{M}, \manifold{M}}
\defeq 
  \inf \set[\Big]{\int_{\manifold{M}}
  \abs{\Deriv U}^p  + F \brk{u}
  \st U \in \sobolev^{1, p} \brk{\manifold{M}, \Rset^\nu} \text{ and }\tr_{\partial \manifold{M}} U= u}
\end{equation}
(Formally, \(\energy^{1, p}_{F}
  \brk{u, \partial \manifold{M}, \manifold{M}} = \energy^{1, p}_{\mathrm{ext}}
  \brk{u, \partial \manifold{M}, \manifold{M}}\) when \(F\) is taken to be \(0\) on \(\manifold{N}\) and \(\infty\) on \(\Rset^\nu \setminus \manifold{N}\).) 
Isobe has showed that
\[
\begin{split}
\tr_{\smash{\manifold{M}'}} \brk{
  &\sobolev^{1, p} \brk{\smash{\manifold{M}'}\times \intvr{0}{1}, \manifold{N}}}
  \\
  &\qquad = 
  \set{u \in \sobolev^{1 - 1/p, p} \brk{\smash{\manifold{M}'}, \manifold{N}}
  \st 
    \varliminf_{\varepsilon \to 0}  \energy^{1, p}_{\dist \brk{\cdot, \manifold{N}}^p/\varepsilon^p} \brk{u, \smash{\manifold{M}'}, \smash{\manifold{M}'} \times \intvr{0}{1}} < \infty}\\
&\qquad   = 
  \set{u \in \sobolev^{1 - 1/p, p} \brk{\smash{\manifold{M}'}, \manifold{N}}
  \st 
  \lim_{L \to 0}\;
  \varlimsup_{\varepsilon \to 0} \energy^{1, p}_{\dist \brk{\cdot, \manifold{N}}^p/\varepsilon^p} \brk{u, \smash{\manifold{M}'}, \smash{\manifold{M}'} \times \intvr{0}{L}} = 0}.
\end{split}
\]
Moreover, the quantities appearing in Isobe’s characterisations are comparable to the extension energy \(  \energy^{1, p}_{\mathrm{ext}} \brk{u, \smash{\manifold{M}'}, \smash{\manifold{M}'} \times \intvr{0}{1}}\).
This characterisation of traces has an extrinsic character, as \eqref{eq_re0iP9aeShiejie5eey3ahxe} defines a penalised energy for extensions that need not satisfy the constraint in \(\manifold{M}\).
  
Mazowiecka and the author \cite{Mazowiecka_VanSchaftingen_2023} have charecterised the traces with a \emph{quantitative topological screening}
in the spirit of Fuglede’s modulus \cite{Fuglede_1957} and of the topological screening for the approximation \cite{Bousquet_Ponce_VanSchaftingen}: for some \(\delta_0 \in \intvo{0}{\infty}\), one has \(u \in \tr_{\smash{\manifold{M}'}} \brk{\sobolev^{1, p} \brk{\manifold{M}' \times \intvr{0}{1}, \manifold{N}}}\) if and only if there exists a summable Borel-measurable function \(w \colon \smash{\manifold{M}'} \to [0,\infty]\)
such that if \(\Sigma\) is a finite homogeneous \(\floor{p}\)-dimensional simplicial complex, if \(\Sigma_0\) is a \(\floor{p - 1}\)-dimensional homogeneous subcomplex of \(\Sigma\) and if \(\sigma \colon  \Sigma \to \smash{\manifold{M}'}\) is a Lipschitz-continuous map satisfying
\begin{align}
\label{eq_ae5oosaejoo3aeB0ipe1OhZ4}
 \seminorm{\sigma}_{\mathrm{Lip}} \sup_{y \in \Sigma} \inf_{z \in \Sigma_0} d^{\Sigma} \brk{y, z} & \le \delta_0, &
 & \text{ and }&
 \int_{\Sigma_0} w \compose \sigma &< \infty,
\end{align}
then there exists a mapping \(V \in \sobolev^{1, p} (\Sigma, \manifold{N})\) such that \(\tr_{\Sigma_0} V   = u \compose \sigma \restr{\Sigma_0}\) and
\begin{equation*}
 \int_{\Sigma} \abs{\Deriv V}^p 
 \leq
  \gamma^\lambda_{\Sigma_0, \Sigma}
   \
   \seminorm{\sigma}_{\mathrm{Lip}}^{p - 1} \int_{\Sigma_0} w \compose \sigma,
\end{equation*}
where \(\lambda > 1\) is fixed,
\begin{equation*}
 \gamma^{\lambda}_{\Sigma_0, \Sigma}
 \defeq \sup_{\substack{z \in \Sigma_0 \\ \delta > 0}} \frac{\abs{B^\Sigma_{\lambda \delta} \brk{z}} }{\delta \abs{\Sigma_0 \cap B^\Sigma_{\delta} \brk{z}}}
\end{equation*} 
and \(B^{\Sigma}_r \brk{x}\) are balls in \(\Sigma\) with respect to the distance \(d^\Sigma\) induced by quotienting the canonical distance on the regular simplex.
Moreover, \(\energy^{1, p}_{\mathrm{ext}} \brk{u, \smash{\manifold{M}'}, \smash{\manifold{M}'} \times \intvr{0}{1}}\) and \(\int_{\smash{\manifold{M}'}} w\) can be taken to be comparable.
This second characterisation is more intrinsic, but requires to consider the whole range of Lipschitz-continuous mappings on simplicial complexes, which is conceptually much more intricate that the  Gagliardo semi-norm defining \(\sobolev^{s, p} \brk{\smash{\manifold{M}'}, \manifold{N}}\) in \eqref{eq_ooJubaebeeh4xae8eoPh0tei} that only considers pairs of points.

The problem of extension of traces is also related to an \emph{approximation problem}:
one has 
\[
\tr_{\smash{\manifold{M}'}}
\brk{ \smash{\smash{\sobolev}^{1, p} \brk{\smash{\manifold{M}'} \times \intvr{0}{1}, \manifold{N}}}}
\subseteq
  \smash{\overline{R^1_{m - \floor{p} - 2} \brk{\smash{\smash{\manifold{M}'}}, \manifold{N}}}}^{\smash{\smash{\sobolev}^{1 - 1/p, p} \brk{\smash{\manifold{M}'}, \manifold{N}}}}
\]
with equality if and only if the homotopy groups \(\pi_1 \brk{\manifold{N}}, \dotsc, \pi_{\floor{p - 1}} \brk{\manifold{N}}\) are finite and, when \(p \in \Nset\), the group
\(\pi_{p - 1}\brk{\manifold{N}}\) is trivial \cite{VanSchaftingen_ExtensionTraces},
where for  \(\ell \in \Zset\) such that \(\ell < m - 1\), \(R^1_\ell \brk{\smash{\manifold{M}'}, \manifold{N}}\) is defined as the set of maps \(u \colon \smash{\manifold{M}'} \to \manifold{N}\) for such that there is  some set \(\Sigma \subseteq \smash{\manifold{M}'}\) which is a finite union of compact subsets of \(\ell\)-dimensional embedded smooth submanifolds for which
\begin{align*}
u &\in C^1 \brk{\smash{\manifold{M}'} \setminus \Sigma, \manifold{N}}&
&\text{and }&
&\sup \, \set[\big]{\dist \brk{x, \Sigma} \abs{\Deriv u \brk{x}} \st x \in \smash{\manifold{M}'}\setminus \Sigma} < \infty.
\end{align*}

In the special case where the domain \(\manifold{M}'\) happens to be simply connected, the theory of liftings of Sobolev mappings over the
universal covering \(\pi \colon \lifting{\manifold{N}} \to \manifold{N}\) can be applied to show that one has \(u \in \tr_{\manifold{M}'} \brk{\sobolev^{1, p} \brk{\manifold{M}' \times \intvr{0}{1}, \manifold{N}}}\) if and only if \(u = \pi \compose \lifting{u}\) for some \(\lifting{u} \in \tr_{\manifold{M}'} \brk{\sobolev^{1, p} \brk{\manifold{M}' \times \intvr{0}{1}, \lifting{\manifold{N}}}}\)
\citelist{\cite{Mironescu_VanSchaftingen_2021_APDE}*{Th.\thinspace{}3}\cite{Bethuel_2014}*{Th.\ 1.5}}; 
when \(2 \le sp < \dim \manifold{M}'\), every \(u \in \sobolev^{s, p} \brk{\manifold{M}', \manifold{N}}\) can be written as \(\pi \compose \lifting{u}\) for some \(\lifting{u} \in \sobolev^{s, p} \brk{\manifold{M}', \lifting{\manifold{N}}}\) if and only if the fundamental group \(\pi_1 \brk{\manifold{N}}\) is finite \citelist{\cite{Bethuel_Chiron_2007}\cite{Mironescu_VanSchaftingen_2021_APDE}\cite{Bourgain_Brezis_Mironescu_2001}\cite{Bethuel_Demengel_1995}\cite{Isobe_2003}*{Th.\ 1.6 \& 1.7}}.
When \(\manifold{N} = \Sset^1\), whose universal covering is \(\pi \colon \Rset \to \Sset^1\), Sobolev liftings allows to reduce the extension of traces to the classical linear problem \citelist{\cite{Brezis_Mironescu_2021}*{Th. 11.2}\cite{Bourgain_Brezis_Mironescu_2004}\cite{Brezis_Mironescu_2001}\cite{Bethuel_Demengel_1995}}.

\medbreak

The main goal of the present work is to establish the local character of the extension of traces on a collar neighourhood.

\begin{theorem}
\label{theorem_local_collar}
Let \(\smash{\manifold{M}'}\) and \(\manifold{N}\) be compact Riemannian manifolds,
let \(\brk{G_i}_{i \in I}\) be a finite covering of \(\smash{\manifold{M}'}\) by open sets and let
\(u \colon \smash{\manifold{M}'} \to \manifold{N}\) be Borel-measurable.
One has 
\[
  u \in \tr_{\smash{\manifold{M}'}}\brk{
 \sobolev^{1, p} \brk{\smash{\manifold{M}'}\times \intvr{0}{1}, \manifold{N}}}
\]
if and only if for every \(i \in I\),
\[u \restr{G_i} \in \tr_{G_i}
 \brk{\sobolev^{1, p} \brk{G_i\times \intvr{0}{1}, \manifold{N}}}.
\]
Moreover, one has
\[
 \energy^{1, p}_{\mathrm{ext}} \brk{u, \smash{\manifold{M}'}, \smash{\manifold{M}'} \times \intvr{0}{1}}
 \le C \sum_{i \in I}  \energy^{1, p}_{\mathrm{ext}} \brk{u \restr{G_i},  G_i, G_i \times \intvr{0}{1}},
\]
where the constant \(C\) only depends on \(\brk{G_i}_{i \in I}\) and \(p\).
\end{theorem}

In particular, \cref{theorem_local_collar} indicates that further characterisations of the traces of Sobolev maps can focus without loss of generality on balls.

The extension is known to satisfy a local linear boundedness principle:
when \(\manifold{M}' = \Bset^{m - 1}\) and \(p < m\), then the equality
\eqref{eq_AhH4EiNohbaesh8oocei2wua} implies the estimate \eqref{eq_ohd0eequ9ahleuhouYi3beiG} \cite{Monteil_VanSchaftingen_2019}.
As a consequence of \cref{theorem_local_collar}, this principle holds for a general extension to a collar neighbourhood:
\begin{theorem}
\label{theorem_UBP}
Let \(\smash{\manifold{M}'}\) and \(\manifold{N}\) be compact Riemannian manifolds and let \(1 < p < \dim \brk{\manifold{M}' \times \intvr{0}{1}}\).
If 
\begin{equation*}
  \tr_{\smash{\manifold{M}'}} \brk{
  \sobolev^{1, p} \brk{\smash{\manifold{M}'}\times \intvr{0}{1}, \manifold{N}}}
  = \sobolev^{1-1/p, p} \brk{\smash{\manifold{M}'}, \manifold{N}},
\end{equation*}
then there is a constant \(C \in \intvo{0}{\infty}\) such that for every \(u \in \sobolev^{1-1/p, p} \brk{\smash{\manifold{M}'}, \manifold{N}}\),
\begin{equation*}
\energy^{1, p}_{\mathrm{ext}} \brk{u, \smash{\manifold{M}'}, \smash{\manifold{M}'} \times \intvr{0}{1}}
\le C \energy^{1 - 1/p, p} \brk{u, \smash{\manifold{M}'}}.
\end{equation*}
\end{theorem}

Even though the statement \cref{theorem_UBP} was already known as a byproduct characterisation of the cases where the trace is surjective, its new proof through \cref{theorem_local_collar} and the nonlinear uniform boundedness for the extension \cite{Monteil_VanSchaftingen_2019} gives some general reason for these linear estimates to go along with the surjectivity of the trace.

Whereas is the linear case \(\manifold{N} = \Rset^\nu\), \cref{theorem_local_collar} has a straightforward proof through a partition of the unity, such a linear construction is not available when \(\manifold{N}\) is a manifold.
Our construction is based instead on a folding procedure, whose essential construction is to note that given two mappings \(U_0, U_1
\in \sobolev^{1, p} \brk{\intvo{0}{1} \times \intvo{0}{1}, \manifold{N}}\)
satisfying the condition
\(
  \tr_{\intvo{0}{1} \times \set{0}} U_0
  = \tr_{\intvo{0}{1} \times \set{0}} U_1
\),
the mapping
\(U_* \colon \intvo{0}{1} \times \intvo{0}{1} \to  \manifold{N}\) defined for every \(\brk{x_1, x_2} \in \intvo{0}{1} \times \intvo{0}{1}\) by
\[
 U_* \brk{x_{1}, x_{2}}
 \defeq
 \begin{cases}
  U_0 \brk{2x_{1}, x_2 - 2 x_{1}} & \text{if \(x_1 \le x_2/2\)},\\
  U_1 \brk{x_2, 2x_{1} - x_2} & \text{if \(x_2/2 < x_{1} \le x_2\)},\\
  U_1 \brk{x_1, x_2} & \text{if \(x_2 < x_{1}\)},
 \end{cases}
\]
also
belongs to  
\(\sobolev^{1, p} \brk{\intvo{0}{1} \times \intvo{0}{1}, \manifold{N}}\) and has its Sobolev energy controlled by those of \(U_0\) and \(U_1\).

As a consequence of the proof, we also get a result on penalised extension energies:

\begin{theorem}
\label{theorem_local_penalised}
Let \(\smash{\manifold{M}'}\) be a compact Riemannian manifol,d let \(F \colon \Rset^\nu \to \intvc{0}{\infty}\) be Borel-measurable,
let \(\brk{G_i}_{i \in I}\) be a finite covering of \(\smash{\manifold{M}'}\) by open sets and let
\(u \colon \smash{\manifold{M}'} \to \manifold{N}\) be Borel-measurable.
One has 
\[
 \energy^{1, p}_{F} \brk{u, \smash{\manifold{M}'}, \smash{\manifold{M}'} \times \intvr{0}{1}}
 \le C \sum_{i \in I}  \energy^{1, p}_{F} \brk{u \restr{G_i},  G_i, G_i \times \intvr{0}{1}},
\]
where the constant \(C\) only depends on \(\brk{G_i}_{i \in I}\) and \(p\).
\end{theorem}

\subsection{Global extension of manifolds}
As observed by Isobe \cite{Isobe_2003}, the problem of global extension of Sobolev mappings from the boundary of a manifold faces an additional \emph{global obstruction}:
if 
\[
 C 
  \brk{\partial \manifold{M}^{\floor{p}}, \manifold{N}}
  \not \subseteq 
  \tr_{\partial \manifold{M}^{\floor{p}}} \brk{C
  \brk{\manifold{M}^{\floor{p}}, \manifold{N}}},
\]
then
 \[   
 \overline{\tr_{\partial \manifold{M}} \vphantom{\dot{W}}\brk{\sobolev^{1, p} \brk{\manifold{M}, \manifold{N}}}}^{\sobolev^{1 - 1/p, p}\brk{\partial \manifold{M}, \manifold{N}}}
 \not \subseteq 
\sobolev^{1 - 1/p, p}\brk{\partial \manifold{M}, \manifold{N}}.
 \]
Here, \(\manifold{M}^{\ell}\) denotes the \(\ell\)-dimensional component of a triangulation of \(\manifold{M}\).

Combined with the obstructions to the extension on a collar neighbourhood, this is the only obstruction to the extension of Sobolev mappings:
 if \(\manifold{M}\) is manifold with compact boundary \(\partial \manifold{M}\) and if \(1 < p < \dim \manifold{M}\), then \cite{VanSchaftingen_ExtensionTraces}
\begin{equation}
\label{eq_aivahbah7muiK9Hia3geeyoy}
  \tr_{\partial \manifold{M}} \brk{
  \sobolev^{1, p} \brk{\manifold{M}, \manifold{N}}}
  = \sobolev^{1-1/p, p} \brk{\partial \manifold{M}, \manifold{N}}.
\end{equation}
if and only if \begin{enumerate*}[label=(\roman*)]  \item  the homotopy groups \(\pi_{1} \brk{\manifold{N}}, \dotsc, \pi_{\floor{p - 2}} \brk{\manifold{N}}\) are all finite, \item the homotopy group \(\pi_{\floor{p - 1}} \brk{\manifold{N}}\) is trivial and \item
\(C 
  \brk{\partial \manifold{M}^{\floor{p}}, \manifold{N}}
  = 
  \tr_{\partial \manifold{M}^{\floor{p}}} \brk{C
  \brk{\manifold{M}^{\floor{p}}, \manifold{N}}}.
\)
  \end{enumerate*}

Again, it is natural to characterise the set \( \tr_{\partial \manifold{M}} \brk{
  \sobolev^{1, p} \brk{\manifold{M}, \manifold{N}}}\) when 
  \eqref{eq_aivahbah7muiK9Hia3geeyoy} fails.
The first contribution in this direction goes back to White 
\cite{White_1988} who has proved that 
\begin{multline*}
  \brk{\tr_{\partial \manifold{M}} \brk{
  \sobolev^{1, p} \brk{\manifold{M}, \manifold{N}}}}
  \cap C \brk{\partial \manifold{M}, \manifold{N}}\\
  = \sobolev^{1-1/p, p} \brk{\partial \manifold{M}, \manifold{N}}
  \cap \set{u \in C \brk{\partial \manifold{M}, \manifold{N}} \st 
  u\vert_{\partial \manifold{M}^{\floor{p}}} \in \tr_{\partial \manifold{M}^{\floor{p}}} C \brk{\manifold{M}^{\floor{p}}, \manifold{N}}}.
\end{multline*}

Isobe \cite{Isobe_2003} has proved that one has  \(u
\in \tr_{\partial \manifold{M}}  \brk{
  \sobolev^{1, p} \brk{\manifold{M}, \manifold{N}}}\) if and only if one has 
  \(u \in \tr_{\partial \manifold{M}}  \brk{
  \sobolev^{1, p} \brk{\partial \manifold{M} \times \intvr{0}{1}, \manifold{N}}}\)
  and \(u \restr{\partial \manifold{M}^{\floor{p - 1}}} \in \tr_{\partial \manifold{M}^{\floor{p - 1}}} C \brk{\manifold{M}^{\floor{p}}, \manifold{N}}\)
  for a generic skeleton (or homotopic in VMO to a map in \(\tr_{\partial \manifold{M}^{p - 1}} C \brk{\manifold{M}^{\floor{p}}, \manifold{N}}\) when \(p \in \Nset\)).
  
The quantitative topological screening characterisation of Mazowiecka and the author also has a global version: for some \(\delta_0 > 0\), \(u
\in \tr_{\partial \manifold{M}}  \brk{
  \sobolev^{1, p} \brk{\manifold{M}, \manifold{N}}}\) if and only if there exists a summable Borel-measurable function \(w \colon \partial \manifold{M} \to \intvc{0}{\infty}\)
such that if \(\Sigma\) is a finite homogeneous \(\floor{p}\)-dimensional simplicial complex, if \(\Sigma_0\subset \Sigma\) is a \(\floor{p - 1}\)-dimensional homogeneous simplicial subcomplex and if
 \(\sigma \colon  \Sigma \to \manifold{M}\) is a Lipschitz-continuous map satisfying
\begin{align}
\label{eq_Pheenahgiegu8Meiphiufaev}
\sigma \brk{\Sigma_0} & \subseteq \partial \manifold{M},&
 \seminorm{\sigma}_{\mathrm{Lip}} \sup_{y \in \Sigma} \inf_{z \in \Sigma_0} d^{\Sigma} \brk{y, z} & \le \delta_0, &
 & \text{ and }&
 \int_{\Sigma_0} w \compose \sigma &< \infty,
\end{align}
then there exists a mapping \(V \in \sobolev^{1, p} (\Sigma, \manifold{N})\) such that \(\tr_{\Sigma_0} V   = u \compose \sigma \restr{\Sigma_0}\) and
\begin{equation*}
 \int_{\Sigma} \abs{\Deriv V}^p 
 \leq
  \gamma^\lambda_{\Sigma_0, \Sigma}
   \
   \seminorm{\sigma}_{\mathrm{Lip}}^{p - 1} \int_{\Sigma_0} w \compose \sigma,
\end{equation*}
where \(\lambda > 1\) is fixed.

Again the global Sobolev extension of traces is related to an approximation problem.
Defining \(R^1_\ell \brk{\manifold{M}, \manifold{N}}\) as the set of mappings 
\(U \colon \manifold{M} \to \manifold{N}\) such that there exists some set \(\Sigma \subseteq \manifold{M}\) which is a finite union of compact subsets of embedded \(\ell\)-dimensional compact submanifolds transversal to \(\partial \manifold{M}\)
such that
\begin{align*}
U &\in C^1 \brk{\manifold{M} \setminus \Sigma, \manifold{N}}&
&\text{and}&
&\sup \set{\dist \brk{x, \Sigma} \abs{\Deriv U \brk{x}}\st x \in \manifold{M} \setminus \Sigma}  < \infty,
\end{align*}
then 
\begin{equation}
 \tr_{\partial \manifold{M}}
\brk{ \smash{\smash{\sobolev}^{1, p} \brk{\manifold{M}, \manifold{N}}}}
\subseteq
 \overline{
  \tr_{\partial \manifold{M}}
 \brk[\big]{
  R^1_{m - \floor{p} - 1} \brk{\smash{\manifold{M}}, \manifold{N}}}}^{\smash{\smash{\sobolev}^{1 - 1/p, p} \brk{\partial \manifold{M}, \manifold{N}}}},
\end{equation}
with equality if and only if the homotopy groups \(\pi_1 \brk{\manifold{N}}, \dotsc, \pi_{\floor{p - 1}} \brk{\manifold{N}}\) are finite and, when \(p \in \Nset\), the homotopy group
\(\pi_{p - 1}\brk{\manifold{N}}\) is trivial \cite{VanSchaftingen_ExtensionTraces}.

Although the global extension problem has an intrinsic topological character, the next result factors the global problem in a collar neighbourhood problem, that can be localised and a purely topological obstruction, as stated by Isobe \cite{Isobe_2003}.

\begin{theorem}
\label{theorem_global}
Let \(\manifold{M}\) be a Riemannian manifold with compact boundary and \(\manifold{N}\) be a compact Riemannian manifold. 
Then, the following are equivalent
\begin{enumerate}[label=(\roman*)]
 \item
 \label{it_phue3au9shah5mo2Eexuu1se}
\(
u \in \tr_{\partial \manifold{M}}\brk{
 \sobolev^{1, p} \brk{\manifold{M}, \manifold{N}}
}\),
\item
\label{it_etho8eatei8xoomea2ahS7vo}
\(
  u \in \tr_{\partial \manifold{M}}\brk{
 \sobolev^{1, p} \brk{\partial \manifold{M} \times \intvr{0}{1}, \manifold{N}}}
\) 
and the set
\begin{equation*}
\begin{split}
 \{\xi \in B_\delta \st u \compose \Pi_{\partial \manifold{M}}\compose \brk{\id_{\partial \manifold{M}^{\floor{p - 1}}} + \xi}
 &\text{ is homotopic in \(\VMO \brk{\partial \manifold{M}^{\floor{p - 1}}, \manifold{N}}\)}\\
 &\text{to some \(V \restr{\partial \manifold{M}^{\floor{p - 1}}}\) with \(V \in C \brk{\manifold{M}^{\floor{p}}, \manifold{N}}\)}\},
\end{split}
\end{equation*}
has positive measure, where \(\partial \manifold{M}\) is isometrically embedded into \(\Rset^\mu\) and \(\Pi_{\partial \manifold{M}} \colon \partial \manifold{M} + B_\delta \to \partial \manifold{M}\) is the corresponding nearest-point retraction,
\item
\label{it_aikai2Mepohn8sui5Kaigha7}
\(
  u \in \tr_{\partial \manifold{M}}
 \sobolev^{1, p} \brk{\partial \manifold{M} \times \intvr{0}{1}, \manifold{N}}
\)
and for every Borel-measurable function \(w \colon \partial \manifold{M} \to \intvc{0}{\infty}\) such that \(\int_{\partial \manifold{M}} w < \infty\), there is a triangulation of \(\manifold{M}\) such that \(u \restr{\partial \manifold{M}^{\floor{p - 1}}}\) is homotopic in \(\VMO \brk{\partial \manifold{M}^{\floor{p - 1}}, \manifold{N}}\) to \(V \restr{\partial \manifold{M}^{\floor{p - 1}}}\) for some \(V \in C \brk{\manifold{M}^{\floor{p}}, \manifold{N}}\).
\end{enumerate}
Moreover, one has then 
\begin{equation}
\energy^{1, p}_{\mathrm{ext}} \brk{u, \partial \manifold{M}, \manifold{M}}
\le \Theta \brk[\big]{\energy^{1, p}_{\mathrm{ext}} \brk{u, \partial \manifold{M}, \partial \manifold{M} \times \intvr{0}{1}}},
\end{equation}
for some convex function \(\Theta \colon \intvr{0}{\infty} \to \intvr{0}{\infty}\) satisfying \(\Theta \brk{0} = 0\) that does not depend on \(u\).
\end{theorem}

A key feature of the topological condition of \ref{it_etho8eatei8xoomea2ahS7vo} and \ref{it_aikai2Mepohn8sui5Kaigha7} is that they essentially prescribe a topological condition on the boundary datum \(u\) at large scales on \(\partial \manifold{M}\).

\section{Folding extensions}

\subsection{Folding mappings across a flat boundary}

Our main tool to prove \cref{theorem_local_collar} is the following proposition that will allow us to transition from one extension to another extension.
The trick to deal with the nonlinear constraint on the target is to perform a suitable \emph{folding} of the domain.

\begin{figure}
\includegraphics{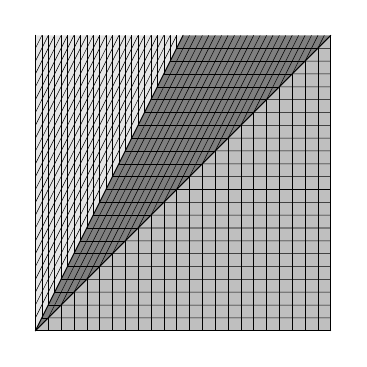}
\caption{The coordinates lines of the changes of variable used in the construction of \(U_*\) by folding in \cref{proposition_folding_local}.}
\end{figure}

\begin{proposition}
\label{proposition_folding_local}
Given manifolds \(\manifold{W}\) and \(\manifold{N}\), and \(U_0, U_1\colon
\manifold{W} \times \intvo{0}{1} \times \intvo{0}{1} \to  \manifold{N}\),
we define
\(U_* \colon \manifold{W} \times \intvo{0}{1} \times \intvo{0}{1} \to  \manifold{N}\) 
for each \(x = \brk{x', x_{m - 1}, x_{m}} \in \manifold{W} \times \intvo{0}{1} \times \intvo{0}{1}\)
by
\begin{equation}
\label{eq_Feeyeeyeepoiro8waiGai6ai}
 U_* \brk{x', x_{m - 1}, x_{m}}
 \defeq
 \begin{cases}
  U_0 \brk{x', 2x_{m - 1}, x_m - 2 x_{m - 1}} & \text{if \(x_{m - 1} \le x_m/2\)},\\
  U_1 \brk{x', x_m, 2x_{m - 1} - x_m} & \text{if \(x_m/2 < x_{m - 1} \le x_m\)},\\
  U_1 \brk{x', x_{m - 1}, x_{m}} & \text{if \(x_m < x_{m - 1}\)}.
 \end{cases}
\end{equation}
If \(p \in \intvr{1}{\infty}\), if \(U_0, U_1
\in \sobolev^{1, p} \brk{\manifold{W} \times \intvo{0}{1} \times \intvo{0}{1}, \manifold{N}}\)
and if
\begin{equation}
\label{eq_oox9jo6piekaiPhiex8iequo}
  \tr_{\manifold{W} \times \intvo{0}{1} \times \set{0}} U_0
  = \tr_{\manifold{W} \times \intvo{0}{1} \times \set{0}} U_1,
\end{equation}
then
\(U_* \in \sobolev^{1, p} \brk{\manifold{W} \times \intvo{0}{1} \times \intvo{0}{1}, \manifold{N}}\),
\begin{gather*}
   \tr_{\manifold{W} \times \intvo{0}{1} \times \set{0}} U_*
   = \tr_{\manifold{W} \times \intvo{0}{1} \times \set{0}} U_0
  = \tr_{\manifold{W} \times \intvo{0}{1} \times \set{0}} U_1,\\
  \tr_{\manifold{W} \times \set{0} \times \intvo{0}{1}} U_*
  = \tr_{\manifold{W} \times \set{0} \times \intvo{0}{1}} U_0\\
  \tr_{\manifold{W} \times \set{1} \times \intvo{0}{1}} U_*
  = \tr_{\manifold{W} \times \set{1} \times \intvo{0}{1}} U_1,
\end{gather*}
and
\begin{equation}
 \int_{\manifold{W} \times \intvo{0}{1} \times \intvo{0}{1}}
 \abs{\Deriv U_*}^p
 \le
 C
 \int_{\manifold{W} \times \intvo{0}{1} \times \intvo{0}{1}}
 \abs{\Deriv U_0}^p +
 \abs{\Deriv U_1}^p,
\end{equation}
where the constant \(C\) only depends on \(p\).
\end{proposition}
\begin{proof}
\begin{figure}
\includegraphics{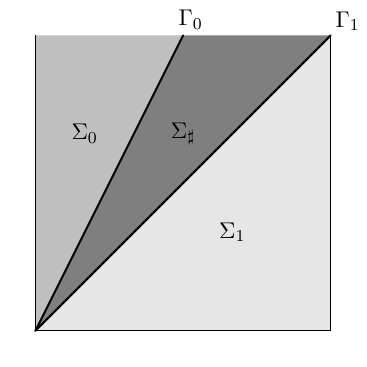}
\caption{The sets \(\Sigma_0\), \(\Sigma_\sharp\), \(\Sigma_1\), \(\Gamma_0\) and \(\Gamma_1\) appearing in the proof of \cref{proposition_folding_local}.}
\end{figure} 
Defining the sets 
\begin{align*}
  \Sigma_0 &\defeq \set{ \brk{y_1, y_2} \in \intvo{0}{1} \times \intvo{0}{1} \st y_1 < y_2/2},\\
  \Sigma_\sharp &\defeq \set{ \brk{y_1, y_2} \in \intvo{0}{1} \times \intvo{0}{1} \st y_2/2 < y_1 < y_2},\\
  \Sigma_1 &\defeq \set{ \brk{y_1, y_2} \in \intvo{0}{1} \times \intvo{0}{1} \st y_2 < y_1},
\end{align*}
we have by invariance under diffeomorphisms of Sobolev spaces
\(U_*\restr{\Sigma_0} \in  \sobolev^{1, p} \brk{\manifold{W} \times \Sigma_0, \manifold{N}}\), 
\(U_*\restr{\Sigma_\sharp} \in  \sobolev^{1, p} \brk{\manifold{W} \times \Sigma_\sharp, \manifold{N}}\) and
\(U_*\restr{\Sigma_\sharp} \in  \sobolev^{1, p} \brk{\manifold{W} \times \Sigma_1, \manifold{N}}\).
Moreover, defining the sets
\begin{align*}
   \Gamma_0 & \defeq \set{ \brk{y_1, y_2} \in \intvo{0}{1} \times \intvo{0}{1} \st y_1 = y_2/2},\\
  \Gamma_1 &\defeq \set{ \brk{y_1, y_2} \in \intvo{0}{1} \times \intvo{0}{1} \st y_1 = y_2},
\end{align*}
we have by definition of \(U_*\) in \eqref{eq_Feeyeeyeepoiro8waiGai6ai}, by the condition \eqref{eq_oox9jo6piekaiPhiex8iequo} and the invariance of trace under change of variables,
\begin{align*}
 \tr_{\Gamma_0} U_*\restr{\Sigma_0}
 &= \tr_{\Gamma_0} U_*\restr{\Sigma_\sharp}
 &\text{and}&
 &
 \tr_{\Gamma_1} U_*\restr{\Sigma_1} &= \tr_{\Gamma_1} U_*\restr{\Sigma_\sharp} .
\end{align*}
The conclusion follows from the gluing property of Sobolev functions having the same trace (see for example \cite{Leoni_2017}*{Th.\ 18.1}).
\end{proof}

\subsection{Folding through an open covering}
We now transfer the folding construction to an open covering of a compact manifold.

\begin{proposition}
\label{proposition_folding_global}
Let \(\smash{\manifold{M}'}\) be a compact Riemannian manifold, \(\brk{G_i}_{i \in I}\) a finite covering of \(\smash{\manifold{M}'}\) and let, for each \(i \in I\), \(U_i \in \sobolev^{1, p} \brk{G_i \times \intvo{0}{1}, \manifold{N}}\).
If for every \(i, j \in I\),
\[
 \tr_{\brk{G_i \cap G_j} \times \set{0}} U_i = \tr_{\brk{G_i \cap G_j} \times \set{0}} u_j,
\]
then
there exists \(U \in \sobolev^{1, p} \brk{G_i \times \intvo{0}{1},\manifold{N}}\) such that for every \(i \in I\),
\[
 \tr_{G_i \times \set{0}} U = \tr_{G_i \times \set{0}} U_i
\]
and
\[
 \int_{\smash{\manifold{M}'} \times \intvo{0}{1}} \abs{\Deriv u}^p
 \le C \sum_{i \in I} \int_{G_i \times \intvo{0}{1}} \abs{\Deriv u_i}^p,
\]
where the constant \(C\) only depends on \(p\) and \(\brk{G_i}_{i \in I}\).
\end{proposition}

\Cref{proposition_folding_global} will follow from \cref{proposition_folding_local} and the following topological lemma.

\begin{lemma}
\label{lemma_cone}
If \(G \subseteq \Rset^m\) is open, if \(F \subseteq \Bar{B}_1\) is closed
and if
\[
 F \cap \partial B_1 \subseteq G,
\]
then there exists an open cone \(C \subseteq \Rset^m\) and some \(r \in \intvo{0}{1}\) such that
\begin{equation}
\label{eq_Em9aequ8Oom3eighohb9ro4w}
F \setminus B_r \subseteq C \cap \brk{\Bar{B}_1 \setminus B_r}
\subseteq G.
\end{equation}
\end{lemma}
\begin{figure}
\includegraphics{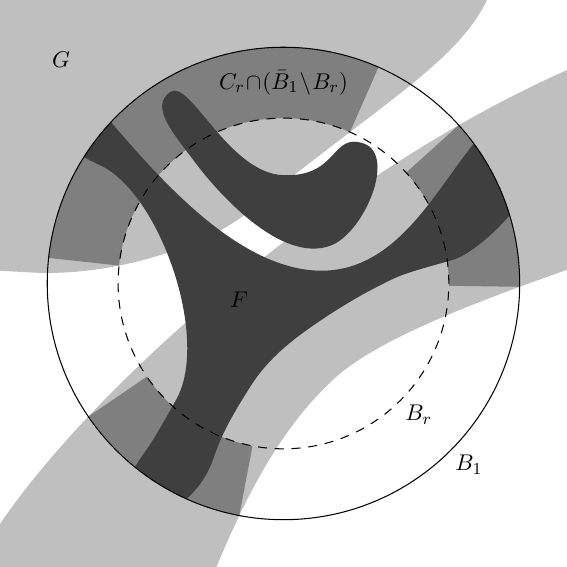}
\caption{The construction of \cref{lemma_cone}: given a open set \(G\) and a closed set \(F\) such that \(F \cap \partial B_1 \subseteq G\), one can find a radius \(r \in \intvo{0}{1}\) so that \(F \setminus B_r\) is contained in a subset \(C \cap \brk{\Bar{B}_1 \setminus B_r}\) of \(G\) where \(C\) is a cone.}
\end{figure}

\begin{proof}
For every \(r \in \intvo{0}{1}\), we define the cone
\[
 C_r
 \defeq
 \set{x \in \Rset^m \st \Rset x \cap \brk{\Bar{B}_1 \setminus B_r} \subseteq G}.
\]
Since the set \(G\) is open, the cone \(C_r\) is also an open set
and satisfies
\[
 C_r \cap \brk{\Bar{B}_1 \setminus B_r} \subseteq G.
\]
If \(0 < r \le s < 1\), one has \(C_r\subseteq C_s\). Moreover, since \(G\) is open
\[
 \bigcup_{r \in \intvo{0}{1}} C_r
 = C_1 \defeq \set{x \in \Rset^m \st \Rset x \cap \partial B_1 \subseteq G}
\]
and \(
 F \cap \partial B_1 \subseteq G \cap \partial B_1
 = C_1 \cap \partial B_1\).
We have thus
\[
 F \subseteq B_1 \cup C_1 = \bigcup_{r \in \intvo{0}{1}} B_r \cup C_r.
\]
By compactness and mononoticity, there exists \(r \in \intvo{0}{1}\), such that
\(
 F \subseteq B_r \cup C_r
\)
and the conclusion \eqref{eq_Em9aequ8Oom3eighohb9ro4w} follows with \(C_r = C\).
\end{proof}

\begin{proof}[Proof of \cref{proposition_folding_global}]
Since the manifold \(\smash{\manifold{M}'}\) is compact, we can assume that for each \(i \in I =  \set{1, \dotsc, K}\), there is an open set \(W_i \subseteq \smash{\manifold{M}'}\) and a diffeomorphism \(\Psi_i \colon W_i \subseteq \manifold{M} \to \Rset^{m - 1}\) such that \(\Bar{G}_i \subseteq W_i\) and
\(\Psi \brk{\Bar{G}_i} = \Bar{B}_1\).

We are going to define inductively over \(i \in \set{1, \dotsc, K}\), open sets \(H_i \subseteq G_i\) and mappings \(V_i \in \sobolev^{1, p} \brk{H_i \times \intvo{0}{1}, \manifold{N}}\)
such that for every \(j \in \set{1, \dotsc, K}\)
\[
  \tr_{\brk{H_i \cap G_j} \times \set{0}} V_i = \tr_{\brk{H_i \cap G_j} \times \set{0}} U_j
\]
and
\begin{equation}
\label{eq_Ahsi0sheiha6goo3kofoh1aF}
H_i \cup \bigcup_{j = i + 1}^K  G_j = \manifold{M}.
\end{equation}
We first set \(H_1 \defeq G_1\) and \(V_1 \defeq U_1\).
Next, assuming that the set \(H_{i - 1}\) and the map \(V_{i - 1}\) have been defined for some \(i \in \set{2, \dotsc, K}\), we consider the closed set
\[
  F_{i} \defeq
  \Psi_{i} \brk[\bigg]{\Bar{G}_{i} \setminus \bigcup_{j  = i + 1}^{K} G_j} = \Bar{B}_1
  \setminus  \bigcup_{j  = i + 1}^{K} \Psi_i \brk{W_i \cap G_j}
  \subseteq \Bar{B}_1
\]
and the open set
\[
  E_i \defeq \Psi_{i} \brk{W_{i} \cap H_{i - 1}}.
\]
\begin{figure}
\includegraphics{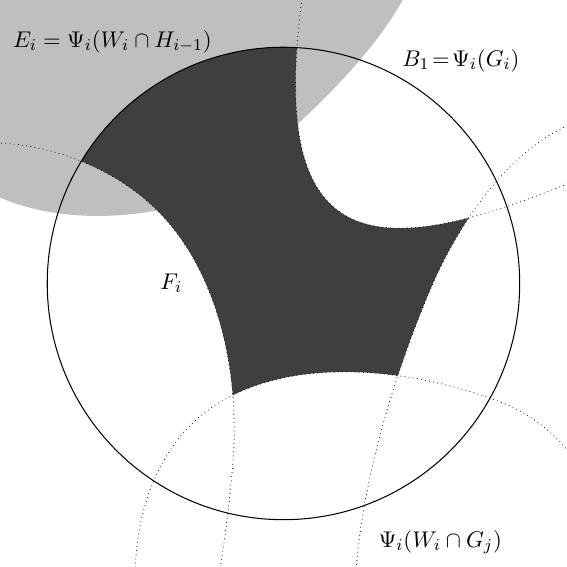}
\caption{The closed set \(F_i =  \Bar{B}_1
  \setminus  \bigcup_{j  = i + 1}^{K} \Psi_i \brk{W_i \cap G_j}\) and the open set \(E_i = \Psi_{i} \brk{W_{i} \cap H_{i - 1}}\) defined in the proof of \cref{proposition_folding_global}.}
\end{figure}%
We have, by \eqref{eq_Ahsi0sheiha6goo3kofoh1aF},
\[
 \partial G_i \subseteq
 W_i \setminus G_i
 \subseteq
 \brk{W_i \cap H_{i - 1}} \cup \bigcup_{j = i + 1}^{K}\brk{W_i \cap G_j},
\]
and thus
\begin{equation}
\label{eq_aMeSux4ga4Jiajee4fik8Gio}
 F_i \cap \partial B_1 =
  \Psi_{i} \brk{\partial G_i} \setminus \bigcup_{j  = i + 1}^{K}
  \Psi_i \brk{W_i \cap G_j}
  \subseteq \Psi_{i}  \brk{W_{i} \cap H_{i - 1}} = E_i.
\end{equation}
In view of \eqref{eq_aMeSux4ga4Jiajee4fik8Gio},
\cref{lemma_cone} applies to the open set \(E_i\) and the closed set \(F_i\) and gives a cone \(C_{i} \subseteq \Rset^m\) and a radius \(r_{i} \in \intvo{0}{1}\) such that
\begin{equation*}
F_i \setminus B_{r_i} \subseteq C_i \cap \brk{\Bar{B}_1 \setminus B_{r_i}}
\subseteq E_i.
\end{equation*}
In particular, we have
\(
  C_i \cap \partial B_1 \subseteq E_i,
\)
and thus
\[
\begin{split}
  \brk{C_i \cap \Bar{B}_{1}} \cup \brk{E_i \setminus \Bar{B}_1}
  &=\brk{C_i \cap B_{1}} \cup \brk{C_i \cap E_i \cap \partial B_1} \cup \brk{E_i \setminus \Bar{B}_1}\\ 
  &= \brk{C_i \cup \brk{G_i \setminus \Bar{B}_1}}
  \cap \brk{\brk{C_i \cap B_1} \cup G_i}
  \end{split}
\]
so that the set \(B_{r_i} \cup \brk{C_i \cap \Bar{B}_{1}} \cup \brk{E_i \setminus \Bar{B}_1}\) is open.

We define the set
\[
 H_{i} \defeq \Psi_{i}^{-1} \brk{B_{r_{i}}}
 \cup \Psi_{i}^{-1} \brk{C_{i} \cap \brk{\Bar{B}_1 \setminus B_{r_{i}}}} \cup \brk{H_{i - 1} \setminus \Psi_{i}^{-1} \brk{\Bar{B}_1}}.
\]
We define \(V_{i} = V_{i - 1}\) on
\(
\brk{H_i \setminus \Psi_{i}^{-1} \brk{\Bar{B}_1}}\),
\(V_{i} = U_{i}\)
on \(\Psi_{i}^{-1} \brk{B_{r_{i}}}\).
On the set \( \Psi_{i}^{-1} \brk{C_{i} \cap \brk{\Bar{B}_1 \setminus B_{r_{i}}}}\) we take \(V_i\) so that \(V_{i} \compose \Psi_i^{-1} \compose \Phi_i^{-1}\)
is the map given by \cref{proposition_folding_local}
to the mappings \(U_i \compose \Psi_i^{-1} \compose \Phi_i^{-1}\)
and \(V_{i -1} \compose \Psi_i^{-1} \compose \Phi_i^{-1}\),
where the mapping
\(\Phi_i \colon \partial B_1 \times \intvo{0}{1} \to B_1 \setminus \Bar{B}_{r_i}\) is defined as 
\[
 \Phi_i \brk{z', z_m} = \brk{1 - z_m r_i} z'.
\]
By the properties of traces, this closes the inductive step.

The estimates can be obtained by observing that the each steps comes with inequalities that only depend on the radii \(r_i\) appearing in the topological argument, and thus only on the covering \(G_1, \dotsc, G_K\) and the charts \(\Psi_1, \dotsc, \Psi_K\),
and are thus independent of the mappings \(U_1, \dotsc, U_K\).
\end{proof}

\begin{proof}
[Proof of \cref{theorem_local_collar}]
If for every \(i \in I\),
\(u \restr{G_i} \in \tr_{G_i}
 \sobolev^{1, p} \brk{G_i\times \intvr{0}{1}, \manifold{N}}
\), then there exists \(U_i \in \sobolev^{1, p} \brk{G_i\times \intvr{0}{1}, \manifold{N}}\) such that \(\tr_{G_i} U_i = u \restr{G_i}\).
In particular, for every \(i, j \in I\),
\[
 \tr_{G_i \cap G_j} U_i = u \restr{G_i \cap G_J}
 = \tr_{G_i \cap G_j} U_j.
\]
Letting \(U \in \sobolev^{1, p} \brk{\manifold{M}'\times \intvr{0}{1}, \manifold{N}}\) be given by \cref{proposition_folding_global}, we have for every \(i \in I\),
\(
  \tr_{\manifold{M}'} U \restr{G_i}
  = u\restr{G_i}
\) and thus \(\tr_{\manifold{M}'} U
  = u\).
\end{proof}

\subsection{Extending penalised energies}

The proof of \cref{theorem_local_penalised} follows the one of \cref{theorem_local_collar}.
The folding construction of \cref{proposition_folding_local}, can be adapted to 

\begin{proposition}
\label{proposition_folding_penal}
Given a manifold \(\manifold{W}\) and \(U_0, U_1\colon
\manifold{W} \times \intvo{0}{1} \times \intvo{0}{1} \to  \Rset^\nu\),
we define
\(U_* \colon \manifold{W} \times \intvo{0}{1} \times \intvo{0}{1} \to  \Rset^\nu\) 
for each \(x = \brk{x', x_{m - 1}, x_{m}} \in \manifold{W} \times \intvo{0}{1} \times \intvo{0}{1}\)
by
\begin{equation*}
 U_* \brk{x', x_{m - 1}, x_{m}}
 \defeq
 \begin{cases}
  U_0 \brk{x', 2x_{m - 1}, x_m - 2 x_{m - 1}} & \text{if \(x_{m - 1} \le x_m/2\)},\\
  U_1 \brk{x', x_m, 2x_{m - 1} - x_m} & \text{if \(x_m/2 < x_{m - 1} \le x_m\)},\\
  U_1 \brk{x', x_{m - 1}, x_{m}} & \text{if \(x_m < x_{m - 1}\)}.
 \end{cases}
\end{equation*}
If \(p \in \intvr{1}{\infty}\), if \(U_0, U_1
\in \sobolev^{1, p} \brk{\manifold{W} \times \intvo{0}{1} \times \intvo{0}{1}, \Rset^\nu}\)
and if
\begin{equation*}
  \tr_{\manifold{W} \times \intvo{0}{1} \times \set{0}} U_0
  = \tr_{\manifold{W} \times \intvo{0}{1} \times \set{0}} U_1,
\end{equation*}
then
\(U_* \in \sobolev^{1, p} \brk{\manifold{W} \times \intvo{0}{1} \times \intvo{0}{1}, \Rset^\nu}\),
\begin{gather*}
   \tr_{\manifold{W} \times \intvo{0}{1} \times \set{0}} U_*
   = \tr_{\manifold{W} \times \intvo{0}{1} \times \set{0}} U_0
  = \tr_{\manifold{W} \times \intvo{0}{1} \times \set{0}} U_1,\\
  \tr_{\manifold{W} \times \set{0} \times \intvo{0}{1}} U_*
  = \tr_{\manifold{W} \times \set{0} \times \intvo{0}{1}} U_0\\
  \tr_{\manifold{W} \times \set{1} \times \intvo{0}{1}} U_*
  = \tr_{\manifold{W} \times \set{1} \times \intvo{0}{1}} U_1,
\end{gather*}
and for every Borel-measurable function \(F \colon \Rset^\nu \to \intvc{0}{\infty}\),
\begin{equation*}
\begin{split}
& \int_{\manifold{W} \times \intvo{0}{1} \times \intvo{0}{1}}
 \abs{\Deriv U_*}^p + F \brk{U_*}\\
 &\qquad \le
 C
 \brk[\bigg]{
 \int_{\manifold{W} \times \intvo{0}{1} \times \intvo{0}{1}}
 \abs{\Deriv U_0}^p + F \brk{U_0}
 +
 \int_{\manifold{W} \times \intvo{0}{1} \times \intvo{0}{1}}
 \abs{\Deriv U_1}^p + F \brk{U_1}},
\end{split}
\end{equation*}
where the constant \(C\) only depends on \(p\).
\end{proposition}

One can then proceed to get the counterpart of \cref{proposition_folding_global}

\begin{proposition}
\label{proposition_folding_penal_global}
Let \(\smash{\manifold{M}'}\) be a compact Riemannian manifold, \(\brk{G_i}_{i \in I}\) a finite covering of \(\smash{\manifold{M}'}\) and let for each \(i \in I\), \(U_i \in \sobolev^{1, p} \brk{G_i \times \intvo{0}{1}, \Rset^{\nu}}\).
If for every \(i, j \in I\),
\[
 \tr_{\brk{G_i \cap G_j} \times \set{0}} U_i = \tr_{\brk{G_i \cap G_j} \times \set{0}} u_j,
\]
then
there exists \(U \in \sobolev^{1, p} \brk{G_i \times \intvo{0}{1},\Rset^\nu}\) such that for every \(i \in I\),
\[
 \tr_{G_i \times \set{0}} U = \tr_{G_i \times \set{0}} U_i
\]
and
\[
 \int_{\smash{\manifold{M}'} \times \intvo{0}{1}} \abs{\Deriv U}^p + F \brk{U}
 \le C \sum_{i \in I} \int_{G_i \times \intvo{0}{1}} \abs{\Deriv U_i}^p + F \brk{U_i},
\]
where the constant \(C\) only depends on \(p\) and \(\brk{G_i}_{i \in I}\).
\end{proposition}

\section{Global results}

\Cref{theorem_global} follows from the characterization of traces of Mazowiecka and the author \cite{Mazowiecka_VanSchaftingen_2023} and from topological screening methods of Bousquet, Ponce and the author \cite{Bousquet_Ponce_VanSchaftingen}.

\begin{proof}[Proof of \cref{theorem_global}]
If \ref{it_phue3au9shah5mo2Eexuu1se} holds,
a restriction to a collar neighbourhood shows that
\(
  u \in \tr_{\partial \manifold{M}}
 \sobolev^{1, p} \brk{\partial \manifold{M} \times \intvo{0}{1}, \manifold{N}}
\).
Following an argument of \cite{Bousquet_Ponce_VanSchaftingen}*{Prop.\ 9.4}, by Mazowiecka and Van Schaftingen’s characterisation of traces of Sobolev spaces \cite{Mazowiecka_VanSchaftingen_2023}*{Th.\ 1.3 and Pr.\ 3.4 } we can assume that \(w\) was chosen so that for every \(\floor{p}\)-dimensional simplicial complex \(\Sigma\), every \(\floor{p - 1}\)-dimensional subcomplex \(\Sigma_0 \subseteq \Sigma\) and every Lipschitz-continuous mapping \(\sigma \colon \Sigma \to \manifold{M}\) satisfying \(\sigma \brk{\Sigma_0}\subseteq \partial \manifold{M} \) and \(\int_{\Sigma_0} w \compose \sigma < \infty\), one has
\(u \restr{\Sigma_0} = \tr_{\Sigma_0} V\) for some \(V \in \sobolev^{1, p} \brk{\Sigma, \manifold{N}}\).
Assuming without loss of generality that \(\manifold{M}\) is embedded in \(\Rset^\mu\) with \(\manifold{M} \subseteq \Rset^\mu_+\) and \(\partial \manifold{M} = \manifold{M} \cap \partial \Rset^\mu_+\) and that there is a smooth \(\Pi_{\manifold{M}} \colon \manifold{U} \to \manifold{M}\) defined on a relative neighbourhood \(\manifold{U}\) of \(\manifold{M}\) in \(\Rset^\mu_+\) such that \(\Pi_{\manifold{M}} \restr{\manifold{M}} = \id\) and \(\Pi_{\manifold{M}} \brk{\manifold{U} \cap \partial \Rset^\mu_+} = \partial\manifold{M}\) \cite{Mazowiecka_VanSchaftingen_2023}*{Prop.\ 3.1}, for \(\delta > 0\) small enough and for \(\xi \in B_{\delta} \subseteq \partial \Rset^\mu_+\), the map \(\sigma_{\xi}
\defeq \Pi_{\manifold{M}} \compose \brk{\id_{\manifold{M}^{\floor{p}}} + \xi} \) is Lipschitz-continuous and satisfies
\[
 \int_{B_{\delta}} \brk[\bigg]{\int_{\partial \manifold{M}^{\floor{p - 1}}} w \compose \sigma_\xi} \dif \xi < \infty,
\]
and taking for \(\ell \in \set{0, \dotsc, \dim \manifold{M}}\) \(\Pi_{\manifold{M}} \brk{\manifold{M}^\ell + \xi}\) provides the required triangulations and \ref{it_etho8eatei8xoomea2ahS7vo} holds.

If \ref{it_etho8eatei8xoomea2ahS7vo} holds, then we have and we are given a Borel-measurable function \(w \colon \partial \manifold{M} \to \intvc{0}{\infty}\) such that \(\int_{\partial \manifold{M}} w < \infty\), then
\[
 \int_{B_{\delta}} \brk[\bigg]{\int_{\manifold{M}^{\floor{p - 1}}} w \compose \sigma_\xi} \dif \xi < \infty.
\]
and \ref{it_aikai2Mepohn8sui5Kaigha7} holds.

If \ref{it_aikai2Mepohn8sui5Kaigha7} holds, by the criterion of Sobolev extension \cite{Mazowiecka_VanSchaftingen_2023}*{Th.\ 1.3}, there exists
\(w \colon \partial \manifold{M} \to \intvc{0}{\infty}\) such that \(\int_{\partial \manifold{M}} w < \infty\) and if \(\Sigma\) is \(\floor{p}\)-dimensional simplicial complex, if \(\Sigma_0\) is a \(\floor{p - 1}\)-dimensional subcomplex,
if \(\sigma \brk{\Sigma} \subseteq \partial \manifold{M}\) and if \(\int_{\Sigma_0} w \compose \sigma < \infty\),
then \(u \compose \sigma\restr{\Sigma_0}\) is homotopic to the restriction of a continuous map from \(\Sigma\) to \(\manifold{N}\).
By assumption, there is a triangulation and \(V \in C \brk{\manifold{M}^{\floor{p}}, \manifold{N}}\) for which
\(u \restr{\partial \manifold{M}^{\floor{p - 1}}}\) is homotopic to \(V \restr{\partial \manifold{M}^{\floor{p - 1}}}\).
It follows  \cite{Bousquet_Ponce_VanSchaftingen}*{proof of Pr.\ 9.11} that if \(\Sigma\) is \(\floor{p}\)-dimensional simplicial complex, if \(\Sigma_0\) is a \(\floor{p - 1}\)-dimensional subcomplex, if \(\sigma \colon \Sigma \to \manifold{M}\) is Lipschitz-continuous and if \(\sigma \brk{\Sigma_0} \subseteq \partial \manifold{M}\), then there exists a Lipschitz-continuous map \(\Phi \colon \partial \manifold{M}^{\floor{p - 1}} \to \manifold{M}\) and a piecewise affine map \(\xi \colon \Sigma \to \manifold{M}^{\floor{p}}\) such that \(\xi \brk{\Sigma_0} \subseteq \partial \manifold{M}^{\floor{p - 1}}\)  and the maps
\(u \compose \sigma\restr{\Sigma_0}\) and \(u \compose \Phi \compose \xi\restr{\Sigma_0}\) are homotopic in \(\VMO \brk{\Sigma_0, \manifold{N}}\).
Since the map \(\xi\) is simplicial, the map
\(u \compose \Phi \compose \xi\restr{\Sigma_0}\) is homotopic to \(V \compose \xi\restr{\Sigma_0} = \brk{F \compose \xi}\restr{\Sigma_0}\) \cite{Bousquet_Ponce_VanSchaftingen}*{Cor.\ 9.10} and \ref{it_phue3au9shah5mo2Eexuu1se} holds.

In order to get the estimate, one observes in \cite{Mazowiecka_VanSchaftingen_2023} that the extension \(\manifold{M}\) is constructed by taking an extension \(U_0 \colon \sobolev^{1, p} \brk{\partial \manifold{M} \times \intvc{0}{1}, \manifold{N}}\) such that \(u = \tr_{\partial \manifold{M}\times \set{0}} U_0\) and \(u_1 \defeq \tr_{\partial \manifold{M}\times \set{1}} U_0 \in \sobolev^{1, p} \brk{\partial \manifold{M}, \manifold{N}}\); following \cite{VanSchaftingen_ExtensionTraces}, one can then find \(U_1 \in \sobolev^{1, p} \brk{\manifold{M}, \manifold{N}}\) such that \(\tr_{\partial \manifold{M}} U_1 = u_1\) and 
\[
   \int_{\manifold{M}} \abs{\Deriv U_1}^p
   \le \Theta \brk[\bigg]{\int_{\partial \manifold{M}} \abs{\Deriv u_1}^p}.
\]
The map \(U\) is then obtained by combining \(U_0\) and \(U_1\) in the collar neighbourhood of \(\partial \manifold{M}\) and satisfies the announced estimate.
\end{proof}

\end{document}